\documentclass{article}

\usepackage{arxiv}
\usepackage{graphicx}
\usepackage[utf8]{inputenc} 
\usepackage[T1]{fontenc}      
\usepackage{amsmath}
\usepackage{amssymb}
\usepackage{amsthm}
\usepackage{graphicx}
\usepackage{color}
\usepackage{hyperref}

\usepackage{lmodern}

\usepackage{mathrsfs}
\usepackage{pgf,tikz,pgfplots}
\usepackage{caption}

\newtheorem{theorem} {Theorem}
\newtheorem{lemma} {Lemma}
\newtheorem{definition} {Definition}

\newtheorem{claim} {Claim}

\newtheorem{prop} {Proposition}

\numberwithin{claim}{section}
\numberwithin{lemma}{section}
\numberwithin{theorem}{section}
\numberwithin{corollary}{section}
\numberwithin{prop}{section}
\numberwithin{proposition}{section}
\numberwithin{observation}{section}
\numberwithin{example}{section}
\numberwithin{definition}{section}

\usepackage{lineno,hyperref}
\modulolinenumbers[5]
\usepackage[utf8]{inputenc} 
\usepackage[T1]{fontenc}    
\usepackage{hyperref}       
\usepackage{url}            
\usepackage{booktabs}       
\usepackage{amsfonts}       
\usepackage{nicefrac}       
\usepackage{microtype}      
\usepackage{lipsum}
\title{A study of a combination of distance domination and resolvability in graphs}

\author{
	Dwi Agustin Retnowardani \\
		Mathematics Departement\\
		 University of Airlangga\\
		 Surabaya, Indonesia \\
		\texttt{2i.agustin@ikipjember.ac.id} \\
		\And
		Muhammad Imam Utoyo \\
		Mathematics Departement\\
		 University of Airlangga\\
		 Surabaya, Indonesia \\
		 \texttt{m.i.utoyo@fst.unair.ac.id}\\
	
	\And
		Dafik \\
		Mathematics Education Department\\
		University of Jember\\ 
		Jember, Indonesia\\
		\texttt{d.dafik@unej.ac.id}\\
		\And
		Liliek Susilowati \\
		Mathematics Departement\\
	 University of Airlangga\\
		 Surabaya, Indonesia \\
		 \texttt{liliek-s@fst.unair.ac.id}\\
		\And 
		Kamal Dliou \\
	National School of Applied Sciences(ENSA)\\
	Ibn Zohr University\\
     Agadir, Morocco \\
	\texttt{dlioukamal@gmail.com}

}
\begin{document}

\maketitle

\begin{abstract}
  For $k \geq 1$, in a graph $G=(V,E)$, a set of vertices $D$ is a distance $k$-dominating set of $G$, if any vertex in $V\setminus D$ is at distance at most $k$ from some vertex in $D$. The minimum cardinality of a distance $k$-dominating set of $G$ is the distance $k$-domination number, denoted by $\gamma_k(G)$. An ordered set of vertices $W=\{w_1,w_2,\ldots,w_r\}$  is a resolving set of $G$, if for any two distinct vertices $x$ and $y$ in $V\setminus W$, there exists $1\leq i\leq r$, such that $d_G(x,w_i)\neq d_G(y,w_i)$. The minimum cardinality of a resolving set of $G$ is the metric dimension of the graph $G$, denoted by $dim(G)$. In this paper, we introduce the distance $k$-resolving dominating set, which is a subset of $V$ that is both a distance $k$-dominating set and a resolving set of $G$. The minimum cardinality of a distance $k$-resolving dominating set of $G$ is called the distance $k$-resolving domination number and is denoted by $\gamma^r_k(G)$. We give several bounds for $\gamma^r_k(G)$ some in terms of the metric dimension $dim(G)$ and the distance $k$-domination number $\gamma_k(G)$. We determine $\gamma^r_k(G)$ when $G$ is a path or a cycle. Afterwards, we characterize the connected graphs of order $n$ having $\gamma^r_k(G)$ equal to $1$, $n-2$, and $n-1$, for $k\geq 2$. Then, we construct graphs realizing all the possible triples $(dim(G),\gamma_k(G),\gamma^r_k (G))$, for all $k\geq 2$. Later, we determine the maximum order of a graph $G$ having distance $k$-resolving domination number $\gamma^r_k(G)=\gamma^r_k\geq 1$, we provide graphs achieving this maximum order for any positive integers $k$ and $\gamma^r_k$. Finally, we establish Nordhaus-Gaddum bounds for $\gamma^r_k(G)$, for $k\geq 2$.
  \end{abstract}
{\bf MSC classes:}{ 05C12, 05C69.}


\section{Introduction}

In this paper, we study finite, simple, and undirected graphs. For graph terminology, we refer to \cite {Gch}.

In 1976, Meir and Moon \cite{meir} studied a combination of two concepts distance and domination in graphs. For $k \geq 1$, we call a  \textit{distance $k$-dominating set} in a graph $G=(V, E)$, a subset of the vertex set $V$ such that for any vertex $v\in V\setminus D$, we have $d_G (v, D) = min \{d_G (v, x): x \in D \} \leq k $, where $d_G(v,x)$ is the distance in $G$ between the vertex $v$ and $x$. The minimum cardinality overall distance $k$-dominating sets of $G$, is the \textit{distance $k$-domination number} and is denoted by $\gamma_k (G)$. When $k=1$, the distance $1$-domination number is the well-known domination number of the graph denoted by $\gamma (G)$. Distance $k$-dominating sets find multiple applications in problems arousing graphs like communication networks \cite{slat}, geometric problems \cite{licht}, facility location problems \cite{henin2}. Results about this well-studied concept can be found surveyed in a recent book chapter~\cite{hening}.

Another concept associated with distance in graphs is resolvability and the metric dimension of graphs, introduced by Harary and Melter \cite{hara} and Slater \cite{slat1}. Let $W =\{w_1, w_2,..., w_r\}$ be an ordered set of vertices in a graph $G$, the \emph{metric representation} of $v$ with respect to $W$ is the $r$-vector $c(v\rvert W) = (d_G(v,w_1),d_G(v,w_2),...,d_G(v,w_r))$. The set $W$ is a \textit{resolving set} of $G$, if for every two distinct vertices $v,u\in V\setminus W$, $c(v\rvert W)\neq c(u\rvert W)$. The minimum cardinality of a resolving set of $G$ is the \textit{metric dimension} of $G$, and is denoted $dim(G)$. Due to their important role either from a theoretical or practical point of view, resolving sets and the metric dimension of graphs attracted attention these past years (see surveys \cite{bail,tilq}). Resolving sets find many applications in several areas like network verification  \cite{beer}, robot navigation \cite{khu}, pharmaceutical chemistry  \cite{char}, coin weighing problems, Mastermind game (see references in \cite{cacer1,her1}) and more.

The problems of finding $\gamma_k (G)$ and $dim(G)$ are both NP-Hard problems in general, see respectively \cite{chang} and \cite{khu}.

To join the utility of resolving sets and distance $k$-dominating sets, we study a set satisfying the two properties. 
\begin{definition}
	A distance $k$-resolving dominating set is a set $S\subseteq V$, where $S$ is both a resolving set and a distance $k$-dominating set of $G$. The distance  $k$-resolving domination number, denoted by $\gamma^r_k (G)$,  is the minimum cardinality of a distance $k$-resolving dominating set of $G$, i.e.,   $\gamma^r_k(G)=min\{\lvert S\rvert: S \text{ is a distance k-resolving dominating set of G}\}$.
\end{definition}
\par A situation where the uses of resolving sets and distance $k$-dominating sets are both needed could represent a possible application of distance $k$-resolving dominating sets. For example, if we consider 
a network of multiple navigation systems, where we need
to control or get information about the positions of the navigation
systems, with a constraint on the distance to transmitters. Then distance $k$-resolving dominating sets are required.

Resolving sets that satisfy additional properties are known and studied. For example, independent resolving set \cite{char2}, is a resolving set that is also an independent set. Connected resolving set \cite{sae}, is a resolving set that is also a connected set.   For $k=1$, the distance $1$-resolving dominating set is a resolving set that is also a dominating set, the minimum cardinality of such set was first studied under the name of resolving domination number in \cite{brig}, while it appeared as metric-location-domination number in \cite{hening3}. More studies were done about that case relating it with other graph parameters, see for example \cite{cacer,gonz,her}. Here we use the name resolving  domination number and denote by $\gamma^r (G)$.

In Section~\ref{sec2}, we give sharp bounds for $\gamma^r_k (G)$ in terms of the metric dimension, the distance $k$-domination number, the order, the diameter, the radius, and the girth of the graph. Also, we give the distance $k$-resolving domination number of the families of paths and cycles. In Section~\ref{sec3}, for all $k\geq 1$, we show that $\gamma^r_k (G)$ is equal to $1$ if and only if $G$ is a path of order at most $k+1$. For $k\geq 2$, we show an equivalence between $\gamma^r_k (G)$ and $dim(G)$, which we use to characterize all graphs of order $n$ having $\gamma^r_k (G)$ equal to $n-1$ and $n-2$. In Section~\ref{sec6}, we determine all the realizable triples of positive integers $(\beta,\gamma,\alpha)$ by a graph $G$ having $dim(G)=\beta$, $\gamma_k(G)=\gamma$, and  $\gamma^r_{k}(G)=\alpha$ when $k\geq 2$, in particular the graphs we construct realizing these values are all trees. In Section~\ref{sec4}, for all $k\geq 1$, we show that a graph $G$ having distance $k$-resolving domination number $\gamma^r_k (G)=\gamma^r_k\geq 1$, has a maximum order of $\gamma^r_k+\gamma^r_k\sum_{p=1}^{k} (2p+1)^{\gamma^r_k-1}$. Also, we construct graphs attaining this maximum order for any arbitrary positive integers $k$ and $\gamma^r_k$. Finally, Section \ref{sec5} is devoted to Nordhaus-Gaddum bounds for the distance $k$-resolving domination number of graphs for $k\geq 2$.

\section{Preliminary results and bounds for $\gamma^r_k (G)$}\label{sec2}

Every superset of a distance $k$-dominating set is a distance  $k$-dominating set. It is true also for resolving sets. This means that every superset of a distance $k$-resolving dominating set is also a distance $k$-resolving dominating set. We give the following bounds that extend bounds given for $k$ equal to $1$ and $2$, in \cite{cacer} and \cite{wardani} respectively to all $k\geq 1$.

\begin{prop}\label{prop2.1}
	Let $G$ be a connected graph. For $k\geq 1$, we have $max\{\gamma_k(G),dim(G)\} \leq \gamma^r_k (G) \leq min\{ \gamma_k(G)+dim(G), n-1\}$.
\end{prop}
\begin{proof}
	Let $S$ be a minimum distance $k$-resolving dominating set of $G$. Since $S$ is both a resolving set and a distance $k$-dominating set, then  $dim(G)\leq \lvert S\rvert$, and $\gamma_k(G)\leq \lvert S\rvert$. Thus $max\{\gamma_k(G),dim(G)\} \leq \gamma^r_k (G)$.\par
	Let $D$ and $W$ be respectively a minimum distance $k$-dominating set and a minimum resolving set of $G$. The set $S=D\cup W$ is a distance $k$-resolving dominating set of cardinality $\lvert S\rvert=\gamma_k(G)+dim(G)$. Also, any subset of $V$ of cardinality $n-1$ is both a resolving set and a distance $k$-dominating set. Then we have $\gamma^r_k (G) \leq min\{ \gamma_k(G)+dim(G), n-1\}$.
\end{proof}
The \textit{eccentricity} of a vertex $v$ in $G$ is the maximum distance between $v$ and any other vertex in $G$. The maximum and minimum eccentricity in $G$ are respectively the \textit{diameter} and the \textit{radius}  of $G$ denoted respectively  $diam(G)$ and $rad(G)$.

\begin{lemma}\label{lem2.1}
	Let $G$ be a connected graph. For $k\geq diam(G)$, $\gamma^r_k(G)=dim(G)$.
\end{lemma}
\begin{proof}
	If  $k\geq diam(G)$, then any non-empty set of vertices in $V$ is a distance $k$-dominating set. Hence any resolving set is also a distance $k$-dominating set of $G$. Therefore, $\gamma^r_k(G)\leq dim(G)$. From Proposition~\ref{prop2.1} it follows that $\gamma^r_k(G)=dim(G)$.
\end{proof}
Let $P_n$ denote the path graph with  $V(P_n)=\{1,2,\ldots,n\}$ and $E(P_n)=\{i(i+1): 1\leq i \leq n-1\}$.  It is proved that $dim(P_n)=1$~\cite{char}, and for $k\geq 1$, $\gamma_k(P_n)=\lceil \frac{n}{2k+1} \rceil $~\cite{davila}. The values of the distance $k$-resolving domination number of $P_n$ for $k$ equal to $1$ and $2$ are given respectively in \cite{brig} and \cite{wardani}. In the following we give $\gamma^r_k(P_n)$ for all $k\geq 1$.
\begin{prop}\label{prop3.1}
	For $k\geq 1$ and $n\geq 2$,
	\begin{displaymath}
	\gamma^r_k(P_n)=\left\{ 
	\begin{array}{ll}
	1 , & {\rm if}\,\,  k\geq n-1,  \\
	2 , & {\rm if}\,\,   \lfloor \frac{n}{2}\rfloor \leq k \leq n-2, \\
	\lceil \frac{n}{2k+1} \rceil , & {\rm if}\,\, 1 \leq k \leq \lfloor \frac{n}{2}\rfloor-1.
	\end{array} \right.
	\end{displaymath}
\end{prop}
\begin{proof}
	In \cite{char}, we have $dim(P_n)=1$. So by Proposition~\ref{prop2.1},  $\gamma_k(P_n)\leq \gamma^r_k(P_n) \leq \gamma_k(P_n)+1$. Also, for $1\leq i,j \leq n$, with $i\neq j$, we have $d_{P_n}(i,j)=\lvert i-j\rvert$. Then any resolving set of cardinality $1$ must be $\{1\}$ or $\{n\}$.   
	\begin{itemize}
		\item 	For $k\geq n-1$, since $diam(P_n)=n-1$, it follows from Lemma~\ref{lem2.1} that $\gamma^r_k(P_n)=dim(P_n)=1$.
		\item 	For $\lfloor \frac{n}{2}\rfloor \leq k \leq n-2$,  based on \cite{davila} $\gamma_k(P_n)=1$, then $\gamma^r_k(P_n)$ is equal to $1$ or $2$. It is clear that an end-vertex is not distance $k$-dominating. Thus, $\gamma^r_k(P_n)= \gamma_k(P_n)+1=2$.
		\item For $1 \leq k \leq \lfloor \frac{n}{2}\rfloor-1$, in \cite{davila} we have $\gamma_k(P_n)=\lceil \frac{n}{2k+1} \rceil \geq 2$. Also, any set $S$ consisting of two or more distinct vertices in $ V(P_n)$ is a resolving set of $P_n$. Thus, $\gamma^r_k(P_n)=\gamma_k(P_n)=\lceil \frac{n}{2k+1} \rceil $.\qedhere
	\end{itemize}
\end{proof}
The path $P_n$ is a graph achieving the bounds in  Proposition~\ref{prop2.1}. For $k\geq n-1$, we have $\gamma^r_k(P_n)=dim(P_n)$. For $1 \leq k \leq \lfloor \frac{n}{2}\rfloor-1$, $\gamma^r_k(P_n)=\gamma_k(P_n)$, and for $\lfloor \frac{n}{2}\rfloor \leq k \leq n-2$, $\gamma^r_k(P_n)=\gamma_k(P_n)+dim(P_n)$.
\par Let $C_n$ denote the cycle graph with $n\geq 3$, where   $V(C_n)=\{0,1,\ldots,n-1\}$ and $E(C_n)=\{i(i+1)  (\text{ mod } n): 0\leq i \leq n-1\}$. We have $dim(C_n)=2$~\cite{char1}, and for $k\geq 1$, $ \gamma_k(C_n)=\lceil \frac{n}{2k+1} \rceil $~\cite{davila}.

\begin{prop}\label{prop3.2}
	For $k\geq 1$ and $n\geq 3$,
	\begin{displaymath}
	\gamma^r_k(C_n)=\left\{ 
	\begin{array}{ll}
	2 , & {\rm if}\,\,  4k+1\geq n,  \\
	3 , & {\rm if}\,\,   4k+2=n, \\
	\lceil \frac{n}{2k+1} \rceil , & {\rm if}\,\, 4k+3\leq n.
	\end{array} \right.
	\end{displaymath}
\end{prop}
\begin{proof}
	We have $d_{C_n}(i,j)=min\{\lvert i-j\rvert,n-\lvert i-j\rvert\}$.
	\begin{claim}\label{claim2.1}
		For $n\geq 2k+2$ and $n\neq 4k+2$, the set of vertices $W=\{0,2k+1\}$ is a resolving set of $C_n$. 
	\end{claim}
	\begin{proof}
		Let $i,j\in V(C_n)\setminus W$, with $i\neq j$. If $d_{C_n}(i,0)\neq d_{C_n}(j,0)$, then $S$  is a resolving set. We suppose that $d_{C_n}(i,0)= d_{C_n}(j,0)$, then either $d_{C_n}(i,0)=i$ and $d_{C_n}(j,0)=n-j$ or $d_{C_n}(i,0)=n-i$ and $d_{C_n}(j,0)=j$. Without loss of generality we suppose that $d_{C_n}(i,0)=i$ and $d_{C_n}(j,0)=n-j$, which means that $i+j=n$. If $d_{C_n}(i,2k+1)= d_{C_n}(j,2k+1)$, then	  
		$min\{\lvert 2k+1-i\rvert,n-\lvert 2k+1-i\rvert\}=min\{\lvert 2k+1-j\rvert,n-\lvert 2k+1-j\rvert\}$. Since $min\{x,y\}=  \frac{x+y- \lvert x-y\rvert}{2}$, then $\lvert n-2\lvert 2k+1-i\rvert \rvert=\lvert n-2\lvert 2k+1-j\rvert \rvert$. 
		
		We suppose that $ n-2\lvert 2k+1-i\rvert =n-2\lvert 2k+1-j\rvert $, which means that $\lvert 2k+1-i\rvert =\lvert 2k+1-j\rvert $. Since  $i\neq j$, then necessarly $2k+1-i = j-2k-1$. It follows that $i+j=4k+2=n$, a contradiction since  $n\neq 4k+2$. 
		
		Otherwise if $ n-2\lvert 2k+1-i\rvert =2\lvert 2k+1-j\rvert-n$, then $n=\lvert 2k+1-i\rvert+\lvert 2k+1-j\rvert$. If $\lvert 2k+1-i\rvert= 2k+1-i$ and $\lvert 2k+1-j\rvert= 2k+1-j$, then $n= 2k+1-i+2k+1-j$. Assuming that $i+j=n$, it means that $n= 2k+1$, a contradiction.
		
		Now if $\lvert 2k+1-i\rvert=i -(2k+1)$ and $\lvert 2k+1-j\rvert=j-(2k+1)$, then $n=i+j-2(2k+1)$. Since  $i+j=n$, it means that $k=0$, a contradiction.
		
		Finally if $\lvert 2k+1-i\rvert=i -(2k+1)$ or $\lvert 2k+1-j\rvert=j-(2k+1)$, we suppose that $\lvert 2k+1-i\rvert=i -(2k+1)$ and $\lvert 2k+1-j\rvert= 2k+1-j$. Then we get that $n=i-j$, again a contradiction. 
		
		It follows that $d_{C_n}(i,2k+1)\neq  d_{C_n}(j,2k+1)$. So for $i,j \in V(C_n)\setminus W$, if $i\neq j$, then $c(i\rvert W)\neq c(j\rvert W)$.
	\end{proof}
	\begin{itemize}
		\item If $2k+1\geq n$, then $ k \geq diam(C_n)$. By Lemma~\ref{lem2.1}, $\gamma^r_k(C_n)=dim(C_n)$. Since $dim(C_n)=2$, then $\gamma^r_k(C_n)=2$.\par 
		If $4k+1\geq n\geq 2k+2$, we have $\gamma^r_k(C_n)\geq dim(C_n)=2$.  
		From Claim~\ref{claim2.1} the set $\{0,2k+1\}$ is a resolving set of $C_n$, it is also a distance $k$-dominating set of $C_n$ for $4k+1\geq n\geq 2k+2$. Therefore $\gamma^r_k(C_n)=2$.  
		\item If $4k+2=n$, based on \cite{davila} we have $\gamma_k(C_{4k+2})=2$, then by Proposition~\ref{prop2.1}, $\gamma^r_k(C_{4k+2})\geq 2$. By using contradiction we suppose that $\gamma^r_k(C_{4k+2})=2$, and let $S$ be a distance $k$-resolving dominating set of cardinality $2$. Since all the vertices have degree $2$, if a vertex $i$ is in a distance $k$-dominating set of cardinality $2$, then the set contains necessarily $i+2k+1 (\text { mod } n )$. Since the cycle $C_n$ is vertex-transitive, we suppose without loss of generality that $S=\{0,2k+1\}$. If we take the vertices $1$ and $4k+1$, then clearly $c(1\rvert S)= c(4k+1\rvert S)$. It follows that $S$ is not a resolving set of $C_{4k+2}$. Hence  $\gamma^r_k(C_{4k+2}) > 2$.

		Now, let us consider the set $S=\{0,1,2k+1\}$, we will show firt that $\{0,1\}\subset S$ is a resolving set of $C_{4k+2}$. For $i\in V(C_n)\setminus S$,  we have $c(i\rvert \{0,1\})=(min\{i,n-i\}, min\{i-1,n-i+1\})$. For $i,j \in V(C_n)\setminus S$, if $c(i\rvert \{0,1\})=c(j\rvert \{0,1\})$, it means that $min\{i,n-i\}=min\{j,n-j\}$ and $min\{i-1,n-i+1\}=min\{j-1,n-j+1\}$. Since  $min\{x,y\}=  \frac{x+y- \lvert x-y\rvert}{2}$, then $\lvert n-2i\rvert=\lvert n-2j\rvert$ and $\lvert n-2(i-1)\rvert=\lvert n-2(j-1)\rvert$. Assuming that $i\neq j$, then necessarly $n-2i=2j-n$ and $n-2(i-1)= 2(j-1)-n$, impossible. Then if $i\neq j$, we have $c(i\rvert \{0,1\})\neq c(j\rvert \{0,1\})$. Therefore $\{0,1\}$ is a resolving set of $C_{4k+2}$. 
		
		Since $\{0,2k+1\}$ is a distance $k$-dominating set of $C_{4k+2}$, it follows that $S=\{0,1,2k+1\}$ is a distance $k$-resolving dominating set of $C_{4k+2}$. Therefore $\gamma^r_k(C_{4k+2})=3$.
		\item If $4k+3\leq n$, in \cite{davila} we have $ \gamma_k(C_n)=\lceil \frac{n}{2k+1} \rceil $. Let us consider the set $S=\{i(2k+1) : 0\leq i\leq \lceil \frac{n}{2k+1} \rceil-1\}$, we have $\lvert S\rvert=\lceil \frac{n}{2k+1} \rceil$.  Claim~\ref{claim2.1} shows that the set $\{0,2k+1\}\subset S$ is a resolving set of $C_n$. Also, it is easy to see that the set $S$ is a  distance $k$-dominating set of $C_n$. It follows that $ \gamma^r_k(C_n)=\lceil \frac{n}{2k+1} \rceil$. \qedhere
	\end{itemize}	
\end{proof}

\begin{prop}\label{prop2.4}
	For $k\geq 1$, let $G$ be a connected graph, such that $rad(G)\leq k$, or $diam(G)=k+1$. Then we have  $dim(G) \leq \gamma^r_k(G)\leq dim(G)+1$.
\end{prop}
\begin{proof}
	Let $G$ be a connected graph with $rad(G)\leq k$. This means that $\gamma_k(G)=1$. Then by Proposition~\ref{prop2.1}, we have $dim(G) \leq \gamma^r_k(G)\leq dim(G)+1$. \par 
	If $diam(G)=k+1$, let $W\subset V$ be a resolving set of $G$. Let $v\in V\setminus dom_k(W)$, where $dom_k(W)=\{v\in V : d_G(v,W)\leq k \}$, then $v$ must be at distance greater or equal to $k+1$ from all the vertices of $W$. Since $diam(G)=k+1$,  the only possible metric representation with respect to $W$ of a vertex $v$ such that $d_G(v,W)\geq k+1$, is a vector having $k+1$ as a value in all its coordinates. Since $W$ is a resolving set, then there is at most one such vertex in $G$. Hence, $dim(G)\leq \gamma^r_k(G)\leq dim(G)+1$.
\end{proof}
For all $k\geq 1$ both bounds in Proposition~\ref{prop2.4} can be achieved. For  $\lfloor \frac{n}{2}\rfloor \leq k \leq n-2$, the graph $P_n$ has $rad(P_n)\leq k$, from Proposition~\ref{prop3.1}, $\gamma^r_k(P_n)=dim(P_n)+1$. From Lemma~\ref{lem2.1}, if $rad(G)\leq diam(G)\leq k$ then for any $G$ we have $\gamma^r_k(G)= dim(G)$. The cycle graphs $C_{2k+2}$ or $C_{2k+3}$ according to Proposition~\ref{prop3.2} are examples of  graphs with $diam(G)=k+1$ having  $\gamma^r_k(G)=dim(G)$. Also from Proposition~\ref{prop3.1}, the path $P_{k+2}$ is a graph of $diam(G)=k+1$ having $\gamma^r_k(G)=dim(G)+1$.
\begin{lemma}\label{lem3.2}\cite{hening2}
	For $k \geq1$, if $G$ is a connected graph of order $n\geq k+1$, and diameter $diam(G)\geq k$. Then there exists a minimum distance $k$-dominating set $D$ of $G$ satisfying for every vertex $v\in D$ there is a vertex $x\in V\setminus D$, such that $d_G(v,x)=k$, and $N_k(x)\cap D=\{v\}$.
\end{lemma}
The following upper bound proved for $dim(G)$ in \cite{brig} is true also for $\gamma^r_k(G)$, the proofs are similar. 
\begin{prop}\label{prop2.3}
	For $k \geq1$, let $G$ be a connected graph of order $n\geq k+1$, with $diam(G)\geq k$. Then $\gamma^r_k(G)\leq n-k\gamma_k(G)$, and this upper bound is achieved for any positive integers $k$ and $\gamma_k(G)$. 
\end{prop}
\begin{proof}
	Suppose that $\gamma_k(G)=\gamma$. Based on Lemma~\ref{lem3.2}, let us consider $D=\{1,2,\ldots,\gamma\}$ a minimum distance $k$-dominating set, such that for all $1\leq i \leq \gamma$, there exists a vertex $w_{i,k}$ verifying that $d_G(i,w_{i,k})=k$, and for $j\neq i$, $d_G(j,w_{i,k})>k$. Now let $P_i=iw_{i,1}w_{i,2}\ldots w_{i,k}$ be a shortest  $(i,w_{i,k})$-path. We can see that for $1\leq p\leq k$, we have $d_G(i,w_{i,p})=p$, and $d_G(j,w_{i,p})>p$. For any two different vertices $w_{i,p}$, $w_{j,q}$, with $1\leq i,j \leq \gamma$, and $1\leq p,q \leq k$, we will check the vector of distances with respect to the set $D$, we discuss the following two cases.
	\begin{itemize}
		\item[$(i)$] If $i\neq j$, we suppose without loss of generality that $q\geq p$. We have $d_G(i,w_{i,p})=p$, and $d_G(i,w_{j,q})\geq q+1>p$.
		\item[$(ii)$] If $i=j$ and $p\neq q$, we have $d_G(i,w_{i,p})=p$, and $d_G(i,w_{i,q})=q\neq p$.
	\end{itemize}   
	It follows that the set $D$ resolves all the vertices $w_{i,p}$, where $1\leq i \leq \gamma$, and  $1\leq p\leq k$. Then the set $S=V\setminus\cup^\gamma_{i=1}\{w_{i,j}\}^k_{j=1}$ is both a distance $k$-dominating set and a resolving set. Hence $\gamma^r_k(G)\leq\lvert S\rvert= \lvert V\setminus\cup^\gamma_{i=1}\{w_{i,j}\}^k_{j=1}\rvert=n-k\gamma=  n-k\gamma_k(G)$.\par 
	The family of Trees $\{T_\gamma: \gamma\geq 1 \}$ illustrated as an example in figure~\ref{fig1} has $\gamma^r_k(T_\gamma)= n-k\gamma$, for $k,\gamma\geq 1$, where $\gamma_k(T_\gamma)=\gamma$. We have any distance $k$-dominating set in $T_\gamma$ must contain at least one vertex in each branch $iw_{i,1}\ldots w_{i,k}$, with $1\leq i\leq \gamma$. Also, the set of vertices $\{1,2,\ldots,\gamma\}$ is a distance $k$-dominating set of $G$. Then clearly $\gamma_k(T_\gamma)=\gamma$.  We can check as above that the set of vertices $\{1,2,\ldots,\gamma\}$ is a resolving set of $T_\gamma$. It follows from Proposition ~\ref{prop2.1} that it is a minimum distance $k$-resolving dominating set of $T_\gamma$ of cardinality $n-k\gamma=n-k\gamma_k(T_\gamma)$.	
	\begin{figure}[h]%
		\centering
		\includegraphics[width=0.6\textwidth]{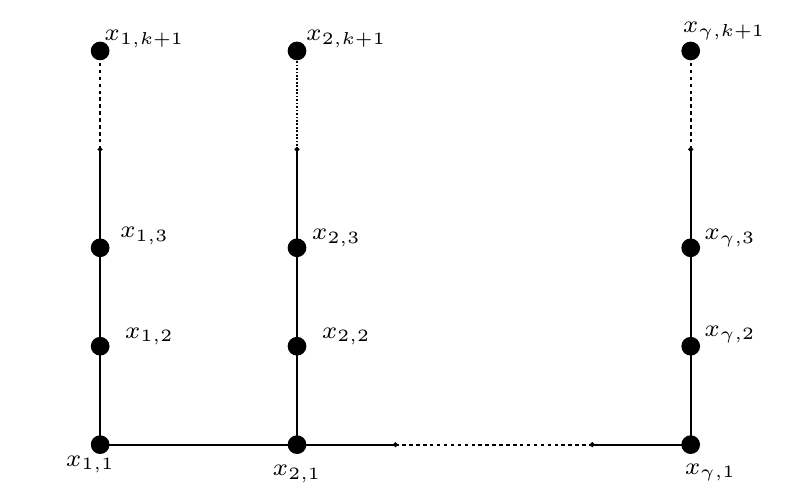}
		\caption{Tree graph $T_\gamma$ having $\gamma^r_k(T_\gamma)= n-k\gamma_k(T_\gamma)$}\label{fig1}
	\end{figure}	
\end{proof} 
For a connected graph $G$ of  order $n$ and  diameter $d$, we have $dim(G)\leq n-d$ \cite{char}. The graphs achieving equality are characterized in \cite{her1}. This type of bound involving the order and the diameter of the graph was provided for the resolving domination number in \cite{cacer}. We give a general upper bound for all $k\geq 1$.
\begin{prop}\label{th4.1}
	For $k\geq 1$, let $G$ be a connected graph of order $n$ and diameter $d$. Then
	\begin{displaymath}
	\gamma^r_k(G)\leq \left\{ 
	\begin{array}{ll}
	n-d , & {\rm if}\,\,  d \leq k,  \\
	n-d+1 , & {\rm if}\,\,  k+1 \leq d \leq 2k, \\
	n-d+\lfloor \frac{d}{2k+1} \rfloor , & {\rm if}\,\,  d \geq 2k+1.
	\end{array} \right.
	\end{displaymath} 
	These bounds are sharp. 
\end{prop} 
\begin{proof}
	Let $P=(0,1,\ldots,d)$ be a diametral path in $G$, i.e., $P$ is a shortest path of lenght $d$. For any two vertices $i$ and $j$ in $P$, we have $d_G(i,j)=\lvert i-j\rvert$.
	\par If $d \leq k$, then by Lemma~\ref{lem2.1}, $\gamma^r_k(G)=dim(G)$. Based on \cite{char}, we have $\gamma^r_k(G)\leq n-d$.
	\par If $k+1 \leq d \leq 2k$, we consider the set of vertices $\{k,d \}$.  For $0 \leq l,m \leq d-1$, with $l\neq m$, we have $d_G(l,d)=\lvert l-d\rvert\neq \lvert m-d\rvert=d_G(m,d)$. Also, for any  $0 \leq l\leq d$,  we have $d_G(l,k)=\lvert l-k\rvert\leq k$. This means that the set  $\{k,d \}$ is resolving and distance $k$-dominating of the vertices $i \notin \{k,d\}$. Now, let $S'=V \setminus \{i: i \notin \{k,d\} \}$, then  $S'$ is a distance $k$-resolving dominating set of $G$. Hence, $\gamma^r_k(G)\leq \lvert S'\rvert=n-d+1$.
	\par If $ d \geq 2k+1$, let us consider the set of vertices $S= \{k,k+(2k+1),\dots,k+j(2k+1),\dots, min\{k+ \lfloor \frac{d}{2k+1} \rfloor (2k+1)  ,d\} \}$.  Let $l$ be a vertex in $P\setminus S$. If $min\{k+\lfloor \frac{d}{2k+1} \rfloor (2k+1)  ,d\}=k+ \lfloor \frac{d}{2k+1} \rfloor (2k+1) $, then either $k+ \lfloor \frac{d}{2k+1} \rfloor (2k+1)  < l \leq d $, or there exists $1\leq i \leq \lfloor \frac{d}{2k+1} \rfloor $ such that $k+(i-1)(2k+1)< l < k+i(2k+1)$, or $0 \leq l < k$. In all those cases there exists a vertex in $S$ at distance less or equal to $k$ from $l$. The same can be observed when $min\{k+\lfloor \frac{d}{2k+1} \rfloor (2k+1)  ,d\}=d$. Furthermore, since $\lvert S\rvert\geq 2$ and for  $0 \leq i,j \leq d$, $d_G(i,j)=\lvert i-j\rvert$, it is straightforward that $S$ resolves the vertices in $P\setminus S$. 
	\par If we consider the set $S'=V\setminus\{ P\setminus S\}$, then $S'$ is a distance $k$-resolving dominating set of the graph $G$. Hence, $\gamma^r_k(G)\leq\lvert S'\rvert= n-d+\lfloor \frac{d}{2k+1} \rfloor$. 
	\par The graph path $P_n$ has diameter $n-1$. From Proposition~\ref{prop3.1} it is a graph achieving the upper bound $n-d$ for $n\leq k+1$. It achieves the upper bound $n-d+1$, when $k+2 \leq n \leq 2k+1$. The path graph $P_n$ also achieves the upper bound $n-d+\lfloor \frac{d}{2k+1} \rfloor$, when $n\geq 2k+2$. 
\end{proof}
If $k=1$, for a connected graph of diameter $d\geq 3$, the upper bound in Proposition~\ref{th4.1} is precisely the bound given in terms of the order and the diameter in \cite{cacer}.

The girth of the graph is the length of a shortest cycle in the graph. The following lower bounds proved in \cite{davila} for  $\gamma_k(G)$ holds also for $\gamma^r_k(G)$ and they are achieved.

\begin{prop}\label{th2.3}
	For $k\geq 1$, let $G$ be a connected graph having diameter $d$, radius $r$, and girth $g$. Then we have
	\begin{itemize}
		\item[$(1)$] $\gamma^r_k(G)\geq \frac{d+1}{2k+1}$;
		\item[$(2)$] $\gamma^r_k(G)\geq \frac{2r}{2k+1}$;
		\item[$(3)$] $\gamma^r_k(G)\geq \frac{g}{2k+1}$, if $g<\infty$.
	\end{itemize}
	These bounds are sharp.
\end{prop}
\begin{proof}
	In \cite{davila}, it is shown that if $G$ is a connected graph of diameter $d$, then $\gamma_k(G)\geq \frac{d+1}{2k+1}$. In the same paper we have if $G$ has radius $r$, then $\gamma_k(G)\geq \frac{2r}{2k+1}$. Also in \cite{davila}, for a connected graph of girth $g<\infty$, we have $\gamma_k(G)\geq \frac{g}{2k+1}$. Since $\gamma^r_k(G) \geq \gamma_k(G)$, the above lower bounds for  $\gamma_k(G)$ are true also for $\gamma^r_k(G)$.\par 
	Some graphs in Proposition~\ref{prop3.1} and~\ref{prop3.2} are examples of graphs attaining these bounds. In $(1)$ consider the path graph of order $n=p(2k+1)$ for $p\geq 2$, since $d=n-1$, we get that $\gamma^r_k(G)= \frac{d+1}{2k+1}$. In $(2)$ consider the path graph of order $n=2p(2k+1)$. We have $r=p(2k+1)$, then from proposition~\ref{prop3.1},  $\gamma^r_k(G)= \frac{2r}{2k+1}$. In $(3)$ take a cycle graph of order $n=p(2k+1)$ for $p\geq 3$, since  $g=n$, then this is a graph having $\gamma^r_k(G)=\frac{g}{2k+1}$.  
\end{proof}

\section{ Graphs with $\gamma^r_k(G)$ equal to $1,n-2,\text{ and } n-1$.}\label{sec3}

Further, let $K_n$ denote \textit{the complete graph} on $n$ vertices, and let $K_{s,t}$ with $s,t\geq 1$ denote \textit{the complete bipartite graph}. For two graphs $G_1$ and $G_2$ the \textit{disjoint union} of $G_1$ and $G_2$, denoted by $G_1\cup G_2$, is the graph with vertex set $V(G_1\cup G_2)=V(G_1)\cup V(G_2)$ and edge set $E(G_1\cup G_2)=E(G_1)\cup E(G_2)$. The  \textit{join graph} of $G_1$ and $G_2$, denoted by $G_1 + G_2$, is the graph obtained from $G_1\cup G_2$  by joining each vertex from $V(G_1)$ to each vertex in $V(G_2)$. We denote by $\overline{G}$ the \emph{complement graph} of $G$. 
\begin{theorem}\label{theo1.1} \cite{char}
	For a connected graph $G$ of order $n\geq2$, we have
	\begin{itemize}
		\item $dim(G)=1$ if and only if $G\cong P_n$.
		\item If $n\geq 4$, then $dim(G)=n-2$ if and only if $G\in \{ K_{s,t}(s,t\geq 1),K_s+\overline{K}_{t}(s\geq 1, t\geq 2),K_s+(K_1\cup K_{t})(s,t\geq 1)  \}$.
		\item $dim(G)=n-1$ if and only if $G\cong K_n$.
	\end{itemize}
\end{theorem} 
In a connected graph $G$ of order $n\geq k+1$, any subset of $V$ of order greater or equal to $n-k$ is a distance $k$-dominating set. 
\begin{lemma}\label{lem2.2}
	Let $k\geq 2$, for $1\leq i\leq k$, if $G$ is a connected graph of order $n\geq i+2$, that is not a path graph, then  $\gamma^r_k(G)=n-i$ if and only if $dim(G)=n-i$.
\end{lemma}
\begin{proof}
	For all $ 1\leq i\leq k$, if $dim(G)=n-i$. Any subset of $V$ of cardinality $n-i\geq n-k$ is a distance $k$-dominating set. Then a resolving set of cardinality $dim(G)=n-i$ is also a distance $k$-dominating set. Therefore $\gamma^r_k(G)=dim(G)=n-i$. \par
	Conversely, if $\gamma^r_k(G)= n-i$, by Proposition~\ref{prop2.1}, we have  $dim(G)\leq n-i$. If $n=i+2$, then $\gamma^r_k(G)= n-i=2$. It follows that $dim(G)$ is equal to $1$ or $2$. Based on Theorem~\ref{theo1.1} the only graphs with $dim(G)=1$ are path graphs, it follows that $dim(G)=2$.
	\par If $n\geq i+3$, we suppose that $dim(G)< n-i$. If $i\leq k-1$, then a resolving set of cardinality $n-(i+1)\geq n-k$ is also a distance $k$-dominating set. Thus $\gamma^r_k(G)\leq n-(i+1)$, impossible.  
	Now if $i=k$, let $W\subseteq V$ be a resolving set of cardinality $n-(k+1)$, and let us denote $1,2,\ldots,{k+1}$ the vertices in $V\setminus W$. Assuming that $\gamma^r_k(G)= n-k$, then there is at least one vertex $v$ in $V\setminus W$ such that $d_G(v,W)=k+1$. Let $w\in W$ be such that $d_G(v,w)=d_G(v,W)= k+1$, and let $Q$ be a shortest $(v,w)$-path. Since $d_G(v,w)=d_G(v,W)$, and $G$ is a connected graph, the only vertex in $W\cap Q$ is $w$. We have $\lvert Q\rvert=k+2$ and $\lvert W\rvert=n-(k+1)$, which means that the subgraph induced by the vertices $1,2,\ldots,{k+1}$ and $w$ is the path $Q$. Without loss of generality, we suppose that the path $Q$ is $(k+1)k\ldots 1w$. Now, let $S= (W\setminus\{w\})\cup \{k\}$. We have $d_G(k,{k+1})=d_G(k,{k-1})=1$, $d_G(k,w)=k\geq 2$, and if $k\geq 3$,  for $1 \leq j \leq k-2$, we have $d_G(k,j)=k-j \geq 2$. Also $d_G({k+1},S\setminus\{k\})\geq k+1$, since $G$ is a connected graph and $n\geq k+3$, then there exists a vertex $u\in S\setminus\{k\}$ such that either or both $1$ and $w$ are adjacent to $u$. This means that $d_G({k-1},u)\leq k$. It follows that $S$ is a resolving set of $G$. Since $d_G(k,i)\leq k$, for $1 \leq i \leq k+1$, $i\neq k$, and $d_G(k,w)=k$,  it means that the set $S$ is also a distance $k$-dominating set of $G$. Hence $\gamma^r_k(G)\leq \lvert S\rvert= n-(k+1)$, a contradiction. Therefore $dim(G)=n-k$.
\end{proof}

By combining Theorem~\ref{theo1.1} and Lemma~\ref{lem2.2} with Proposition~\ref{prop3.1}, we give the following characterizations.
\begin{theorem}\label{th3.1}
	For any graph $G$ of order $n\geq2$, the following statements hold.
	\begin{itemize}
		\item[$(a)$] For all $k \geq 1$, $\gamma^r_k(G)=1$ if and only if $G\in\{P_i\}^{k+1}_{i=2}$.
		\item[$(b)$] If $G$ is a connected graph of order $n\geq 4$,  $\gamma^r_2(G)=n-2$ if and only if $G\in \{P_4, K_{s,t}(s,t\geq 1), K_s+\overline{K}_{t}(s\geq 1, t\geq 2), K_s+(K_1\cup K_{t})(s,t\geq 1)\}$.
		For all $k\geq 3$, $\gamma^r_k(G)=n-2$ if and only if $G\in \{ K_{s,t}(s,t\geq 1), K_s+\overline{K}_{t}(s\geq 1, t\geq 2), K_s+(K_1\cup K_{t})(s,t\geq 1)\}$.
		\item[$(c)$] If $G$ is a connected graph, for any $k\geq 2$,  $\gamma^r_k(G)=n-1$ if and only if $G\cong K_n$.
	\end{itemize}  
\end{theorem} 
\begin{proof}
	\leavevmode
	\begin{itemize}
		\item[$(a)$] For $k\geq 1$, if $\gamma^r_k(G)=1$, then $G$ is a connected graph and from Proposition~\ref{prop2.1}, $dim(G)=1$. The equivalence is completed by Theorem~\ref{theo1.1} and Proposition~\ref{prop3.1}. \par 
		\item[$(b)$] If $G$ is a connected graph of order $n\geq 4$ different from a path graph, then  by Lemma~\ref{lem2.2} we have $\gamma^r_k(G)=n-2$ if and only if $dim(G)=n-2$. Which means by Theorem~\ref{theo1.1} that it is equivalent to $G\in \{K_{s,t}(s,t\geq 1),K_s+\overline{K}_{t}(s\geq 1, t\geq 2),K_s+(K_1\cup K_{t})(s,t\geq 1) \}$. From Proposition~\ref{prop3.1}, we have $\gamma^r_k(P_n)=n-2$, it occurs only when $k=2$ and $n=4$. Then $\gamma^r_2(G)=n-2$ if and only if $G\in \{P_4, K_{s,t}(s,t\geq 1), K_s+\overline{K}_{t}(s\geq 1, t\geq 2), K_s+(K_1\cup K_{t})(s,t\geq 1)\}$. Also, for
		$k\geq 3$, $\gamma^r_k(G)=n-2$ if and only if $G\in \{ K_{s,t}(s,t\geq 1), K_s+\overline{K}_{t}(s\geq 1, t\geq 2), K_s+(K_1\cup K_{t})(s,t\geq 1)\}$.\par 
		\item[$(c)$ ] The only connected graphs of order $2$ and $3$ are respectively $K_2$ and $P_3$ or $K_3$. For $k\geq 2$, from Proposition~\ref{prop3.1} and Theorem~\ref{theo1.1}, we have $\gamma^r_k(K_2)=1$, $\gamma^r_k(P_3)=1$, and $\gamma^r_k(K_3)=2$. If $G$ has order $n\geq 4$, then by Lemma~\ref{lem2.2} and Theorem~\ref{theo1.1}, we have $\gamma^r_k(G)=n-1$ if and only if $G\cong K_n$. \qedhere
	\end{itemize} 
\end{proof}
For $k=1$, we have $\gamma^r(G)=n-1$ if and only if $G\in \{ K_{1,n-1}, K_n \}$ \cite{brig,hening3}. The graphs having $\gamma^r(G)$ equal to $2$ and $n-2$ are fully determined in \cite{cacer} and \cite{hening3} respectively.

\section{Realizable values for $dim(G)$, $\gamma_k(G)$, and $\gamma^r_k(G)$.}\label{sec6}

In Proposition~\ref{prop2.1}, we have $max\{\gamma_k(G),dim(G)\} \leq \gamma^r_k (G) \leq \gamma_k(G)+dim(G)$. For $k=1$, in  \cite{cacer} it is shown that for any three positive integers $\beta$, $\gamma$, and $\alpha$, verifiying that $max\{\gamma,\beta\}\leq \alpha \leq \gamma+\beta$, and  $(\beta,\gamma,\alpha)\notin \{(1,\gamma,\gamma+1) : \gamma\geq 2 \}$. There is always a graph $G$ having $dim(G)=\beta$, $\gamma(G)=\gamma$, and $\gamma^r(G)=\alpha$. We give a similar result for $dim(G)$, $\gamma_k(G)$, and $\gamma^r_k(G)$, for all $k\geq 2$. 
\par The graph families we provide in Theorem~\ref{th5.2} are all trees. To determine $\gamma^r_k(G)$ of some of these graphs, we will need the next formula for the metric dimension of trees that appeared in  \cite{char,hara,slat1}. We will recall some terminology given in \cite{char}. In a tree $T$ 
for $v\in V$,  if the degree $deg(v)\geq 3$, then $v$ is called a \emph{major vertex}. A leaf $l$ i.e. a vertex of degree one, in $T$ is a \emph{terminal vertex} of a major vertex $v$, if $v$ is the closest major vertex in terms of distance to $l$, i.e. for $u$ a major vertex in $T$ different from $v$, we have $d_T(v,l)<d_T(u,l)$.  If $v$ is a major vertex having at least one terminal vertex, then $v$ is called an \textit{exterior major vertex}. Let $L(T)$ and $EX(T)$ denote respectively the number of leaves and the number of exterior major vertices in a tree $T$.
\begin{theorem}\label{th5.1} \cite{char,hara,slat1}
	If $T$ is a tree that is not a path graph, then $dim(T)=L(T)-EX(T)$. Also, any resolving set of $T$ must contain at least one vertex from each branch at an exterior major vertex containing its terminal vertices with at most one exception.
\end{theorem}
\begin{theorem}\label{th5.2}
	For any three positive integers $\beta$, $\gamma$, and $\alpha$, such that $max\{\gamma,\beta\}\leq \alpha \leq \gamma+\beta$, and $(\beta,\gamma,\alpha)\notin \{(1,\gamma,\gamma+1) : \gamma\geq 2 \}$. For all $k\geq 2$, there always exists a tree graph $T$ having $dim(T)=\beta$, $\gamma_k(T)=\gamma$, and $\gamma^r_k(T)=\alpha$. There is no graph realizing the triples $\{(1,\gamma,\gamma+1) : \gamma\geq 2 \}$.
\end{theorem}
\begin{proof}
	Let $\beta,\gamma,\alpha \geq 1$, such that $max\{\gamma,\beta\}\leq \alpha \leq \gamma+\beta$. We discuss the possible values for the triple $(dim(G),\gamma_k(G),\gamma^r_k(G))=(\beta,\gamma,\alpha)$, according to the following cases.
	\begin{itemize}
		\item If $\beta=1$, then $\gamma\leq \alpha \leq \gamma +1$. Also by Theorem~\ref{theo1.1} we have the path graphs are the only graphs having the metric dimension equal to $1$. For $k\geq 2$,  in a path graph any subset of vertices of order greater or equal to $2$ is a resolving set. Then if $\gamma\geq 2$, we have $\alpha=\gamma$. This means that the triple $(1,\gamma,\gamma+1)$ is not realizable by any graph for $\gamma\geq 2$. Also, according to Proposition~\ref{prop3.1} the path graphs realizes the following cases: $(i)$ if $ k+1 \geq n$, we have $\gamma=\beta=\alpha=1$. $(ii)$ If $\lfloor \frac{n}{2}\rfloor \leq k \leq n-2$, then $\gamma=\beta=1$ and $\alpha=2=\gamma+\beta$. $(iii)$ If $1 \leq k \leq \lfloor \frac{n}{2}\rfloor-1$, then $\beta=1< \gamma=\alpha=\lceil \frac{n}{2k+1} \rceil \geq 2$.
		\item If $\gamma=1$, for any $\beta \geq 2$, we have $\beta\leq \alpha \leq \beta +1$. The star graph $K_{1,\beta+1}$ has $\gamma_k(K_{1,\beta+1})=1$, and from Theorem~\ref{theo1.1} and Theorem~\ref{th3.1} we have  $\gamma^r_k(K_{1,\beta+1})=dim(K_{1,\beta+1})=\beta$, for any $k\geq 2$. This means that for $k\geq 2$, the triple $(\beta,1,\beta)$ is realized for all $\beta \geq 2$. For the case of the triple $(\beta,1,\beta+1)$, we consider the spyder tree graph, denoted by $S_{\beta+1,k}$, having one vertex $v_0$ of degree $\beta+1$ with $\beta+1$ leaves $l_i$,  $1\leq i \leq \beta+1$, at distance $k$ from $v_0$. Note that all the vertices of $S_{\beta+1,k}$ are of degree less or equal to $2$ except $v_0$. Clearly $\gamma_k(S_{\beta+1,k})=1$, and based on Theorem~\ref{th5.1}, we have $dim(S_{\beta+1,k})=\beta$. Also any resolving set must contain at least one vertex in all but one of the $(v_0,l_i)$-paths, where $1\leq i \leq \beta+1$. By using contradiction, we suppose that $\gamma^r_k(S_{\beta+1,k})=\beta$. From Theorem~\ref*{th5.1}, we consider that a minimum distance $k$-resolving dominating set $W$ of $S_{\beta+1,k}$ having cardinality $\beta$ contains one vertex in any of the $(v_0,l_i)$-paths, with $1\leq i \leq \beta$. We have the vertex $l_{\beta+1}$ is at distance greater than $k$ from the vertices in $W$. This means that $W$ is not a distance $k$-dominating set, a contradiction. Hence, $\gamma^r_k(S_{\beta+1,k})=\beta+1$.  
		\item If $\beta \geq 2$, and $\gamma \geq 2$, with $max(\gamma,\beta)\leq \alpha \leq \gamma+\beta$. The realizable values for the triple $(\beta,\gamma,\alpha)$ are considered depending on the following five subcases.
		\item[$(i)$] If $2\leq \beta=\gamma < \alpha$, the trees $T^1=\{T^1_{k,m,l}:\, m\geq 0,\, l\geq 1,\, k\geq 2\}$ in  Figure~\ref{fig2} illustrates graphs realizing this case.
		\begin{claim}
			We have $\gamma_k(T^1_{k,m,l})=dim(T^1_{k,m,l})=m+l$, and $\gamma^r_k(T^1_{k,m,l})=m+2l$.
		\end{claim}
		\begin{proof}
			Suppose that $\gamma_k(T^1_{k,m,l})=\gamma$, $dim(T^1_{k,m,l})=\beta$, and $\gamma^r_k(T^1_{k,m,l})=\alpha$. It is clear that $\{v_i \}^{m}_{i=1} \cup \{w_i \}^{l}_{i=1}$ is a minimum distance $k$-dominating set. Then $\gamma=m+l$. Based on Theorem~\ref{th5.1}, we have $\beta=m+l$, and for each $1\leq i \leq l$, a resolving set must contain one vertex from the set of vertices $\{v_{i,j} \}^{k}_{j=0}$. Also, for each  $1\leq i \leq m$, a resolving set must contain one vertex from the set of vertices $\{w_{i,j},w'_{i,j} \}^{k}_{j=0}$. Now, let $S$ be a minimum distance $k$-resolving dominating set of cardinality $\alpha$. We suppose without loss of generality, that $S$ contain a vertex from each $\{v_{i,j} \}^{k}_{j=0}$ with $1\leq i \leq m$, and one vertex from each $\{w_{i,j} \}^{k}_{j=0}$ with $1\leq i \leq l$. Since $d_G(w_i,w'_{i,k})=k$, and for $x\notin \{ w_i,w'_{i,j}\}$ we have $d_G(x,w'_{i,k})>k$. Then to be a distance $k$-dominating set $S$ must contain for each $1\leq i \leq l$, at least $w_i$ or a vertex in $\{w'_{i,j} \}^{k}_{j=0}$. Hence $\alpha \geq m+2l$. It is easy to check that the set of vertices $\{v_{i,1} \}^{m}_{i=1} \cup \{w_{i,k} \}^{l}_{i=1}\cup \{w_{i} \}^{l}_{i=1}$ is a distance $k$-resolving dominating set. Thus $\alpha \leq m+2l$. It follows that $\alpha = m+2l$.			
			\begin{figure}[h]%
				\centering
				\includegraphics[width=1\textwidth]{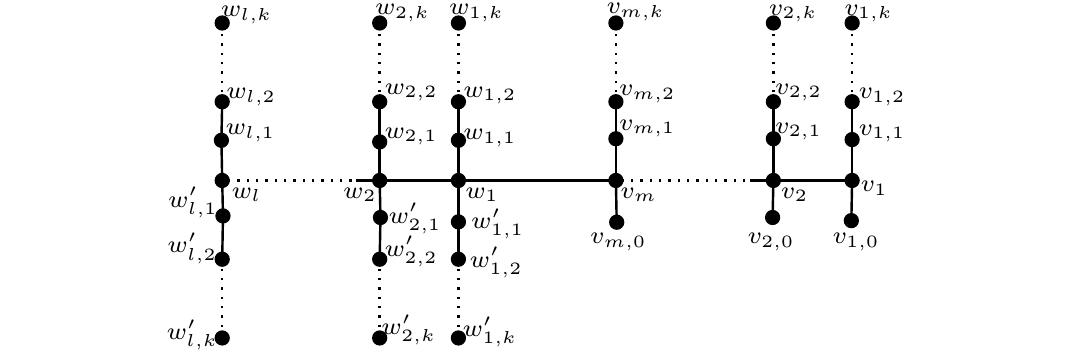}
				\caption{Tree $T^1_{k,m,l}$}\label{fig2}
			\end{figure}		
		\end{proof}
		\item[$(ii)$] If $2\leq \gamma \leq \beta= \alpha$, the family of trees  $T^2=\{T^2_{k,m,l}:\, m\geq 1,\, l\geq 1,\, k\geq 2\}$ represented in Figure~\ref{fig3} gives graphs realizing this case.
		\begin{claim}
			We have $\gamma_k(T^2_{k,m,l})=m+1$, and $\gamma^r_k(T^2_{k,m,l})=dim(T^2_{k,m,l})=m+l$.
		\end{claim}
		\begin{proof}
			The set of vertices $\{v_i \}^{m}_{i=1} \cup \{w \}$ is a minimum distance $k$-dominating set of cardinality $m+1$. Hence $\gamma_k(T^2_{k,m,l})=m+1$. Based on Theorem~\ref{th5.1}, $dim(T^2_{k,m,l})=m+l$. It is straightforward to check that the set of vertices $\{v_{i,1} \}^{m}_{i=1} \cup \{w_{i,1} \}^{l}_{i=1}$ is a distance $k$-resolving dominating set of cardinality $m+l$. Therefore $\gamma^r_k(T^2_{k,m,l})=dim(T^2_{k,m,l})=m+l$.			
		\end{proof}
		\begin{figure}[h]%
			\centering
			\includegraphics[width=0.9\textwidth]{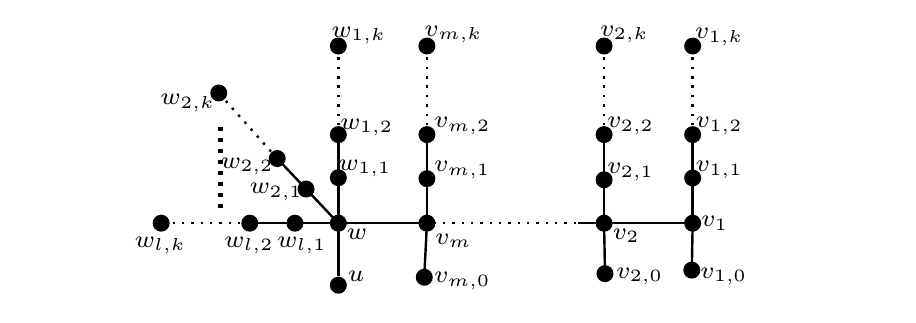}
			\caption{Tree $T^2_{k,m,l}$}\label{fig3}
		\end{figure}
		
		\item[$(iii)$] If $2\leq \beta < \gamma = \alpha$, the family of trees $T^3=\{T^3_{k,m,l}:\, m\geq 1,\, l\geq 1,\, k\geq 2\}$ represented in Figure~\ref{fig4} realizes this case.
		\begin{claim}
			We have $dim(T^3_{k,m,l})=m+1$, and $\gamma^r_k(T^3_{k,m,l})=\gamma_k(T^3_{k,m,l})=m+l+1$.
		\end{claim}
		\begin{proof}
			From Theorem~\ref{th5.1}, we have $dim(T^3_{k,m,l})=m+1$. Also  $\gamma_k(T^3_{k,m,l})=m+l+1$ ($\{v_{i} \}^{m}_{i=1} \cup \{w_{i} \}^{l}_{i=1} \cup \{u \}$ is a minimum distance $k$-dominating set of $T^3_{k,m,l}$). It is easy to check that $\{v_{i,1} \}^{m}_{i=1} \cup \{w_{i} \}^{l}_{i=1} \cup \{u_1 \}$ is a distance $k$-resolving dominating set of cardinality $m+l+1$. It follows that $\gamma^r_k(T^3_{k,m,l})=\gamma_k(T^3_{k,m,l})=m+l+1$.			
		\end{proof}	
		
		\begin{figure}[h]%
			\centering
			\includegraphics[width=0.9\textwidth]{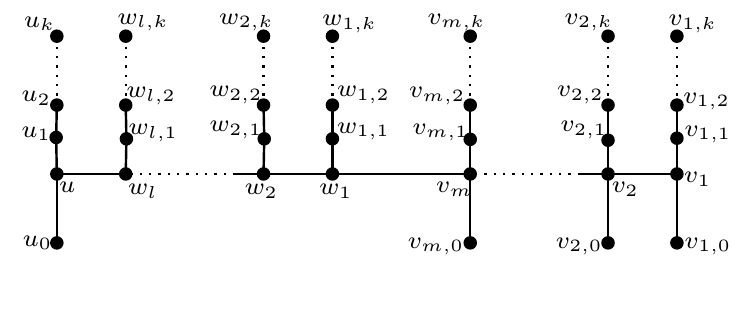}
			\caption{Trees $T^3_{k,m,l}$}\label{fig4}
		\end{figure}
		
		\item[$(iv)$] If $2\leq \gamma < \beta < \alpha$, the family of trees $T^4=\{T^4_{k,m,l,r}:\, m\geq 0,\, l\geq 0,\, r\geq 3,\, k\geq 2\}$ represented in  Figure~\ref{fig5} illustrates graphs realizing this case, where $(m,l)\neq (0,0)$.
		\begin{claim}
			We have  $dim(T^4_{k,m,l,r})=m+l+r-1$, $\gamma_k(T^4_{k,m,l,r})=m+l+1$, and $\gamma^r_k(T^4_{k,m,l,r})=m+2l+r$.
		\end{claim}
		\begin{proof}
			By using similar arguments as for the above claims, we can show that the set of vertices $\{v_{i,k} \}^{m}_{i=1} \cup \{w_{i,k} \}^{l}_{i=1}  \cup \{u_{i,k} \}^{r-1}_{i=1}$ is a minimum resolving set of cardinality $m+l+r-1$.
			The set of vertices $\{v_i \}^{m}_{i=1} \cup \{w_i \}^{l}_{i=1} \cup \{u \}$ is a minimum distance $k$-dominating set of cardinality $m+l+1$. The set of vertices $\{v_{i,1} \}^{m}_{i=1} \cup \{w_{i,k},w_i \}^{l}_{i=1}  \cup \{u_{i,k} \}^{r-1}_{i=1} \cup \{u \}$ is a minimum distance $k$-resolving dominating set of cardinality $m+2l+r$.			
		\end{proof}
		
		\begin{figure}[h]%
			\centering
			\includegraphics[width=0.9\textwidth]{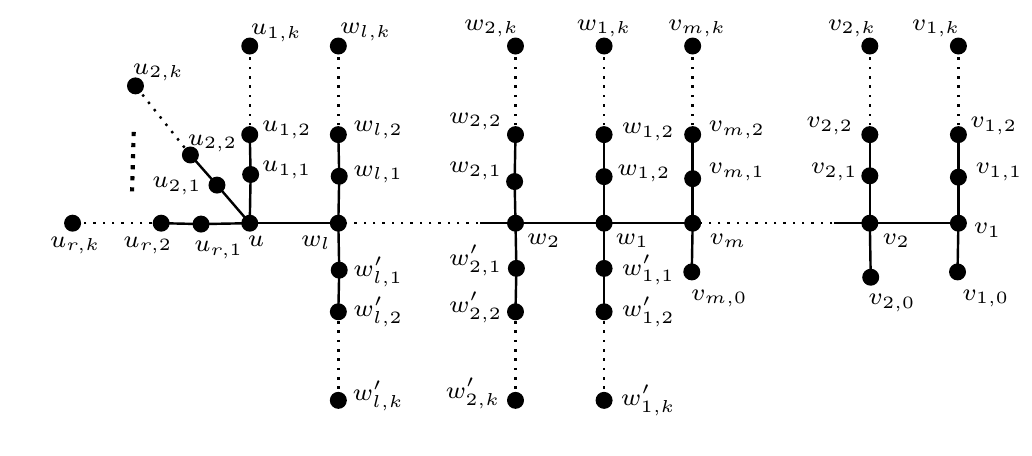}
			\caption{Tree $T^4_{k,m,l,r}$}\label{fig5}
		\end{figure}
		
		\item[$(v)$] If $2\leq \beta < \gamma < \alpha$, Figure~\ref{fig6} illustrates a family of trees $T^5=\{T^5_{k,m,l,r}:\, m\geq 0,\, l\geq 0,\, r\geq 2,\, k\geq 2\}$ realizing this case, where $(m,l)\neq (0,0)$.
		\begin{claim}
			We have $dim(T^5_{k,m,l,r})=m+l+1$, $\gamma_k(T^5_{k,m,l,r})=m+l+r$,  and $\gamma^r_k(T^5_{k,m,l,r})=m+2l+r+1$.
		\end{claim}
		\begin{proof}
			The set of vertices $\{v_{i,k} \}^{m}_{i=1} \cup \{w_{i,k} \}^{l}_{i=1}  \cup \{u_{r,k} \}$ is a minimum resolving set of cardinality $m+l+1$. The set of vertices $\{v_i \}^{m}_{i=1} \cup \{w_i \}^{l}_{i=1} \cup \{u_i \}^{r}_{i=1}$ is a minimum distance $k$-dominating set of cardinality $m+l+r$. The set of vertices $\{v_{i,1} \}^{m}_{i=1} \cup \{w_{i,k} \}^{l}_{i=1} \cup \{w_i \}^{l}_{i=1}  \cup \{u_{i} \}^{r}_{i=1} \cup \{u_{r,k} \}$ is a minimum distance $k$-resolving dominating set of cardinality $m+2l+r+1$.			
		\end{proof}	
		\begin{figure}[h]%
			\centering
			\includegraphics[width=0.9\textwidth]{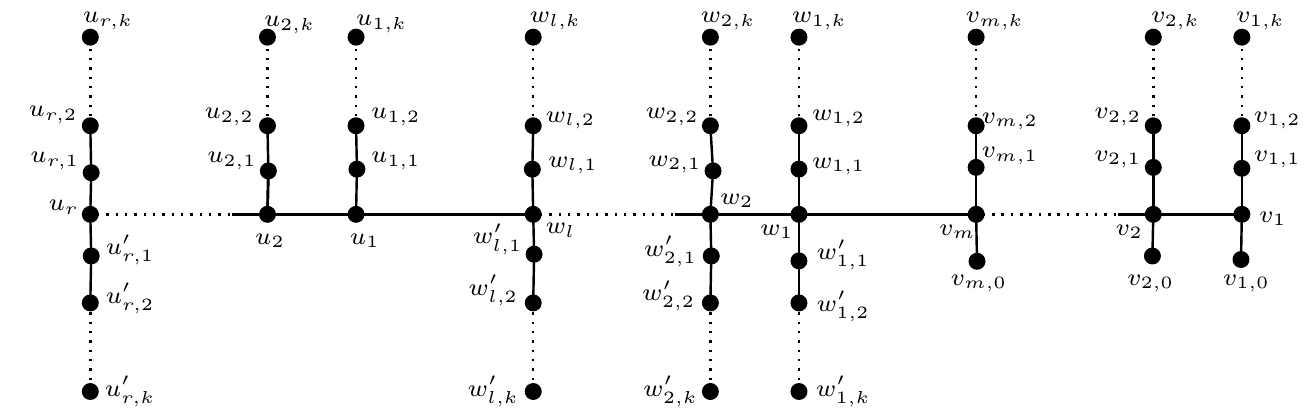}
			\caption{Tree $T^5_{k,m,l,r}$}\label{fig6}
		\end{figure}		
	\end{itemize} 
\end{proof}
\section{Maximum order graphs}\label{sec4}
The maximum order $n$ of a graph $G$  having diameter $d$ and metric dimension $dim(G)=\beta$, was shown to be $\beta + d^\beta$ \cite{char,khu}. This was proved by considering the maximum possible number of distinct metric representations with respect to a minimum resolving set. But this maximum order is only achieved when $d\leq 3$ or $\beta=1$. Later, Hernando et al. \cite{her1} proved a stronger result by showing that $n\leq (\lfloor \frac{2d}{3} \rfloor +1)^\beta+\sum_{i=1}^{\lceil \frac{d}{3} \rceil} (2i+1)^{\beta-1}$, where the maximum order is achieved for any arbitrary positive integers $d$ and $\beta$.  

Cáceres et al. \cite{cacer}  showed that for a graph $G$ of order $n$  having  $\gamma^r(G)=\gamma^r$, then $n\leq \gamma^r+\gamma^r\cdot 3^{\gamma^r-1}$. They also provided graphs achieving this maximum order. Next, we generalize this result for $\gamma^r_k(G)$ for all $k\geq 1$.
\begin{theorem}\label{th4.2}
	For $k\geq 1$, the maximum order of a connected graph $G$ having distance $k$-resolving domination number $\gamma^r_k$ is $\gamma^r_k+\gamma^r_k\sum_{p=1}^{k} (2p+1)^{\gamma^r_k-1}$. This maximum order is achieved for any $k,\gamma^r_k\geq 1$.
\end{theorem} 
\begin{proof}
	Let $G$ be a graph of order $n$ and let $S$ be a minimum distance $k$-resolving dominating set of $G$. For any vertex $x\in V\setminus S$, let us consider $v_i$ a vertex in $S$, such that $d_G(x,v_i)=p\leq k$. If $\gamma^r_k(G)=\gamma^r_k\geq 2$, for any vertex $v_j$ from $S$ different from $v_i$, the triangle inequality gives $\lvert d_G(x,v_j)-d_G(v_i,v_j)\rvert\leq d_G(x,v_i)=p$. It follows that the metric representation of $x$ with respect to $S$ has the coordinate corresponding to $v_i$ equal to $p$ and for the other coordinates there are at most $2p+1$ possible values in each of the other $\gamma^r_k-1$ coordinates. Therefore, there are at most $(2p+1)^{\gamma^r_k-1}$ possible metric representations of $x$ with respect to the set $S$. Since $1\leq p \leq k$, there are at most $\sum_{p=1}^{k} (2p+1)^{\gamma^r_k-1}$ distinct metric representations for the vertices at distance less or equal to $k$ from $v_i$. Since  $\lvert S\rvert=\gamma^r_k$, then $n \leq \gamma^r_k+\gamma^r_k\sum_{p=1}^{k} (2p+1)^{\gamma^r_k-1}$. \par Let $k$ and $\gamma^r_k$ be two arbitrary positive integers, we will prove that there exists a graph having distance $k$-resolving domination number $\gamma^r_k$ and order $\gamma^r_k+\gamma^r_k\sum_{p=1}^{k} (2p+1)^{\gamma^r_k-1}$.
	\par If $\gamma^r_k=1$, then from Theorem~\ref{th3.1} the graph $G$ is a path graph of maximum order $k+1$, which coincides with the maximum order bound. If $\gamma^r_k=r\geq 2$, we consider the following subsets of $\mathbb{Z}^r$, $$ Q_0={\scriptstyle\{(0,2k+1,2k+1,\ldots,2k+1),(2k+1,0,2k+1,\ldots,2k+1),\ldots,(2k+1,2k+1,\ldots,2k+1,0)\}}.$$ For all $1\leq i \leq r$,	
	$$Q_i=\{(q_1,q_2,\ldots,q_r): 1\leq q_i\leq k \text{, and for $j\neq i$, } 2k-q_i+1\leq q_j\leq 2k+q_i+1 \}.$$  
	Let $G_r$ be the graph whose vertex set is $V(G_r)=\cup_{i=0}^r Q_i$. For which two vertices $q=(q_1,q_2,\ldots,q_r)$ and $q'=(q'_1,q'_2,\ldots,q'_r)$ are adjacent if and only if $\lvert q_j-q'_j\rvert\leq 1$, for each $1\leq j \leq r$.
	\begin{claim}
		The graph $G_r$ is a connected graph.
	\end{claim}
	\begin{proof}
		If $q_{i,0},q_{j,0}\in Q_0$,  where $q_{i,0}$ has the $i$-th element equal to $0$ and $q_{j,0}$ has the $j$-th element equal to $0$, we construct a $(q_{i,0},q_{j,0})$-path as following $$ {\scriptstyle(2k+1,\dots,\underset{i}{0},2k+1,\dots,\underset{j}{2k+1},\dots,2k+1)(2k+1,\dots,\underset{i}{1},2k+1,\dots,\underset{j}{2k},2k+1,\dots,2k+1)}$$ $${\scriptstyle(2k+1,\ldots,\underset{i}{2},2k+1,\dots,\underset{j}{2k-1},2k+1,\dots,2k+1)\dots\dots}  $$ $$ {\scriptstyle(2k+1,\dots,\underset{i}{k},2k+1,\dots,\underset{j}{k+1},2k+1,\dots,2k+1)(2k+1,\dots,\underset{i}{k+1},2k+1,\dots,\underset{j}{k},2k+1,\dots,2k+1)}$$ $${\scriptstyle(2k+1,\dots,\underset{i}{k+2},2k+1,\dots,\underset{j}{k-1},2k+1,\dots,2k+1)\dots\dots}$$ $${\scriptstyle(2k+1,\dots,\underset{i}{2k},2k+1\dots,\underset{j}{1},2k+1\dots,2k+1)(2k+1,\dots,\underset{i}{2k+1},2k+1,\dots,\underset{j}{0},2k+1,\dots,2k+1)}$$ 
		\par Also, for each $1\leq i \leq r$, if $q=(q_1,q_2,\dots,q_r)\in Q_i$. It is easy to see from the definition of the adjacency in $G_r$, that there is a $(q,q_{i,0})$-path. Hence, the graph $G_r$ is a connected graph.
	\end{proof}
	For $1\leq i \leq r$, and $q\in V(G_r)\setminus Q_0$, we denote $L_i(q)=(f_i(q_1),f_i(q_2),\ldots,f_i(q_r))$, where $f_i$ is an integer-valued function defined as following.
	\par If $q=(q_1,q_2,\ldots,q_r)\in Q_s$, with $s\neq i$.
	\begin{itemize}
		\item For $j\notin \{s,i\}$, $f_i(q_j)=\begin{cases}
		q_j  & if\,\,  q_j=2k+1,\\
		q_j-1 & if\,\,  q_j>2k+1, \\
		q_j+1 & if\,\,  q_j<2k+1.\\
		\end{cases}  $
		\item $f_i(q_s)=\begin{cases}
		q_s & if\,\,  q_s=k,\\
		q_s+1 & if\,\, q_s<k \text{ or } q_i=k+1.\\
		\end{cases}$
		\item $f_i(q_i)=q_i-1$.
	\end{itemize}
	\par If $q=(q_1,q_2,\ldots,q_r)\in Q_i$.
	\begin{itemize}
		\item For $j\neq i$, $f_i(q_j)=\begin{cases}
		q_j & if\,\,  q_j=2k+1,\\
		q_j-1 & if\,\,  q_j>2k+1, \\
		q_j+1 & if\,\,  q_j<2k+1.\\
		\end{cases}$
		\item $f_i(q_i)=q_i-1$.
	\end{itemize}
	\par For $t\geq 1$, we define $L^t_i(q)$ with $L^1_i(q)=L_i(q)$. For $t\geq 2$, $L^t_i(q)=L_i(L^{t-1}_i(q))=(f^t_i(q_1),f^t_i(q_2),\ldots,f^t_i(q_r))$, where $f^t_i$ is the $t$-th iterated function of $f_i$, i.e., $f^t_i=\underbrace{f_i\circ f_i\circ\ldots \circ f_i}_{\text{t times}}$.
	
	\begin{claim}\label{claim2}
		For all $1\leq i \leq r$, for any vertex $q=(q_1,q_2,\ldots,q_r)\in V(G_r)\setminus Q_0$, we have $L_i(q)\in V(G_r)$. Also, $L_i(q)$ is adjacent in $G_r$ to $q$, and  $L^{q_i}_i(q)=q_{0,i}$.
	\end{claim}  
	\begin{proof}
		\par Let $q=(q_1,q_2,\ldots,q_r)\in V(G_r)\setminus Q_0$. For $1\leq i \leq r$, we have\\ $L_i(q)=(f_i(q_1),f_i(q_2),\ldots,f_i(q_r))$. If $q\in Q_s$, where $s\neq i$, for $j\neq s$, we have $2k-q_s+1\leq q_j\leq 2k+q_s+1$, and $1\leq q_s\leq k$. We discuss the membership of $L_i(q)$ according to the following cases.
		\begin{itemize}
			\item[$(i)$] If $q_s< k$, we have $f_i(q_s)=q_s+1\leq k$, $f_i(q_i)=q_i-1\geq 2k-q_s$,  and for $j\notin\{i,s\}$,  $2k-q_s+2\leq f_i(q_j)\leq 2k+q_s$.  Then $L_i(q)=(f_i(q_1),f_i(q_2),\ldots,f_i(q_r)) \in Q_s$.
			\item[$(ii)$] If $q_s=k$, and $q_i> k+1$. Then $f_i(q_i)=q_i-1\geq k+1$, $f_i(q_s)=k$, and for $j\notin\{i,s\}$, $k+1\leq f_i(q_j)\leq 3k+1$. So $L_i(q)=(f_i(q_1),f_i(q_2),\ldots,f_i(q_r)) \in Q_s$.
			\item[$(iii)$] If $q_i=k+1$, then $q_s=k$. It follows that $f_i(q_i)=k$, $f_i(q_s)=k+1$, and for $j\notin\{i,s\}$, $k+1\leq f_i(q_j)\leq 3k+1$. Therefore, $L_i(q) \in Q_i$.
		\end{itemize} 
		\par Now, if $q\in Q_i$, from the definition of $f$ it is easy to see that $L_i(q)\in Q_i$. Hence, for any vertex $q\in V(G_r)\setminus Q_0$, we have $L_i(q)\in V(G_r)$. Moreover, for $q\in V(G_r)\setminus Q_0$, and all $1\leq i,j \leq r$, we have $\lvert f_i(q_j)-q_j\rvert\leq 1$, $f^{q_i}_i(q_i)=0$, and for $j\neq i$, $f^{q_i}_i(q_j)=2k+1$. Thus, $L_i(q)q\in E(G_r)$, and $L^{q_i}_i(q)=q_{0,i}$.  
	\end{proof}
	\begin{claim}
		For all $1\leq i \leq r$, for any vertex $q=(q_1,q_2,\ldots,q_r)\in V(G_r)\setminus Q_0$, $d_{G^r}(q,q_{0,i})=q_i$.
	\end{claim}
	\begin{proof}
		Based on Claim~\ref{claim2} for $1\leq i \leq r$, we have  $qL_i(q)L^2_i(q)\ldots L^{q_i}_i(q)=q_{0,i}$ is a $(q,q_{0,i})$-path in $G_r$ of  length $q_i$. Hence $d_{G^r}(q,q_{0,i})\leq q_i$. Since $q_{0,i}$ and $q$ are vertices having respectively $0$ and $q_i$ at the $i$-th coordinate and any two vertices in $G_r$ can be adjacent only if the difference between the respective coordinates is at most $1$. It follows that $d_{G^r}(q,q_{0,i})\geq q_i$. Therefore, $d_{G^r}(q,q_{0,i})=q_i$.
	\end{proof}  
	From above we can conclude that for any two different vertices $q$ and $q'$ in $V(G_r)\setminus Q_0$, there exists $1\leq i \leq r$, such that $d_{G^r}(q,q_{0,i})=q_i\neq d_{G^r}(q',q_{0,i})=q'_i$. It follows that the set of vertices $Q_0$ is a resolving set of $G_r$. Also, for all $1\leq i \leq r$, and any vertex $q\in Q_i$,  $d_{G^r}(q,q_{0,i})=q_i\leq k$. Hence, the set $Q_0$ is as well a distance $k$-dominating set of $G_r$. Hence, $\gamma^r_k(G_r)\leq \lvert Q_0\rvert=r$.
	\par Suppose that $\gamma^r_k(G_r)\leq r-1$. We have the order of the graph $G_r$ is $\lvert G_r\rvert=r+r\sum_{p=1}^{k} (2p+1)^{r-1}$. Also the maximum order of a graph having $\gamma^r_k(G_r)\leq r-1$, was previously proved to be less or equal to $\gamma^r_k(G_r)+\gamma^r_k(G_r)\sum_{p=1}^{k} (2p+1)^{\gamma^r_k(G_r)-1}\leq (r-1)+(r-1)\sum_{p=1}^{k} (2p+1)^{r-2}$, it is a contradiction. Therefore, $\gamma^r_k(G_r)= r$. 
\end{proof}
For $k=1$, the maximum order in Theorem~\ref{th4.2} is precisely the maximum order given in \cite{cacer}.

\section{Nordhaus-Gaddum type bounds}\label{sec5}
Nordhaus-Gaddum bounds are sharp bounds on the sum or the product of a parameter of a graph $G$ and its complement $\overline{G}$. The survey \cite{aou} contains a bibliography of these types of bounds for some graph parameters. Hernando et al. \cite{her} found Nordhaus-Gaddum type of bounds for the metric dimension and the resolving domination number.  We provide those bounds for the distance $k$-resolving domination number for $k\geq 2$.
\begin{theorem}\label{th6.1}
	For any graph $G$ of order $n\geq 2$,
	\begin{itemize}
		\item if $k=2$, then
		$$3 \leq \gamma^r_2(G)+\gamma^r_2(\overline{G})\leq 2n-1 \text{ and } 2 \leq \gamma^r_2(G)\cdot\gamma^r_2(\overline{G})\leq n(n-1).$$ 
		The lower bounds are attained  if and only if $G\in \{K_2,\overline{K}_2,P_3,\overline{P}_3\}$.\\
		The upper bounds are  attained if and only if $G\in \{K_n,\overline{K}_n\}$.
		\item If $k\geq 3$, then $$ 2 \leq \gamma^r_k(G)+\gamma^r_k(\overline{G})\leq 2n-1 \text{ and } 1 \leq \gamma^r_k(G)\cdot \gamma^r_k(\overline{G})\leq n(n-1).$$ 
		The lower bounds are attained if and only if $G\cong P_4$.\\
		The upper bounds are attained if and only if $G\in \{K_n,\overline{K}_n\}$.
	\end{itemize}
\end{theorem}
\begin{proof}
	If $k=2$, we have from Theorem~\ref{th3.1} $(a)$, $\gamma^r_2(G)=1$ if and only if $G$ is $K_2$ or $P_3$. Also, for any other graph $G$, we have $\gamma^r_2(G)\geq 2$. This means that $\gamma^r_2(G)+\gamma^r_2(\overline{G})\geq 3$ and $ \gamma^r_2(G)\cdot\gamma^r_2(\overline{G})\geq 2$. Since $\gamma^r_2(\overline{K}_2)=2$, and $\gamma^r_2(\overline{P}_3)=2$, we can conclude that these lower bounds are attained if and only if $G\in \{K_2,\overline{K}_2,P_3,\overline{P}_3\}$. \par 
	If $k\geq 3$, based on Theorem~\ref{th3.1} $(a)$,  we have $\gamma^r_k(G)=1$ if and only if $G\in\{P_2,P_3,\ldots, P_{k+1}\}$. The graph $P_4$ is a self-complementary graph, i.e., $\overline{P}_4\cong P_4$, we have $\gamma^r_k(\overline{P}_4)=\gamma^r_k(P_4)=1$. Also, $P_4$ is the only graph whose complement is also a path and has a distance $k$-resolving domination number equal to $1$. Therefore  $\gamma^r_k(G)+\gamma^r_k(\overline{G})\geq 2$ and $ \gamma^r_k(G)\cdot \gamma^r_k(\overline{G})\geq 1$, also these lower bounds are achieved if and only if $G$ is $P_4$.
	\par Otherwise, for $k\geq 2$, we have  $\gamma^r_k(G)=n$ if and only if $G$ is the empty graph on $n$ vertices $\overline{K}_n$, whose complement graph is the complete graph $K_n$. According to Theorem~\ref{th3.1} $(c)$, we have  $\gamma^r_k(K_n)=n-1$. Therefore, for any graph $G$ of order $n\geq 2$, for $k\geq 2$, we have $\gamma^r_k(G)+\gamma^r_k(\overline{G})\leq 2n-1$ and $\gamma^r_k(G)\cdot \gamma^r_k(\overline{G})\leq n(n-1)$. Moreover, these upper bounds are achieved if and only if $G\in \{K_n,\overline{K}_n\}$. 
\end{proof}
Let $G$ be a connected graph with $V(G)=\{1,2,\ldots, n\}$. The graph $G[H^i]$ is the graph obtained from $G$ by replacing the vertex $i$ with a graph $H$ and joining each vertex of $H$ to every vertex adjacent to $i$ in $G$. Let $H_1$ and $H_2$ be two graphs, the graph $G[H_1^i,H_2^j]$ is the graph obtained from $G$ by replacing the vertex $i$ (respectively, $j$) with the graph $H_1$ (respectively, $H_2$) and joining each vertex of $H_1$ (respectively, $H_2$) to every vertex adjacent to $i$ (respectively, $j$) in $G$. If $i$ and $j$ are adjacent in $G$, join every vertex of $H_1$ to every vertex of $H_2$. The Bull graph $B$ is the graph with vertex set $V(B)=\{1,2,3,4,5\}$, and edge set  $V(B)=\{12,13,23,14,25\}$. The graph $B$ is a self-complementary graph, i.e., $\overline{B}\cong B$.

\begin{theorem}
	If $G$ and $\overline{G}$ are both connected graphs of order $n\geq 4$,
	\begin{itemize}
		\item if $k=2$, then $$4 \leq \gamma^r_2(G)+\gamma^r_2(\overline{G})\leq 2n-4 \text{ and } 4 \leq \gamma^r_2(G)\cdot\gamma^r_2(\overline{G})\leq (n-2)^2.$$ 
		The upper bounds are attained if and only if $G\cong P_4$.
		\item If $k\geq 3$, then $$2 \leq \gamma^r_k(G)+\gamma^r_k(\overline{G})\leq 2n-6 \text{ and } 1 \leq \gamma^r_k(G)\cdot \gamma^r_k(\overline{G})\leq (n-3)^2.$$ 
		The lower bounds are  attained if and only if $G\cong P_4$.\\
		The upper bounds are  attained if and only if $G\in \{P_4,C_5,B\}\cup\\ \{P_4[K_{n-3}^1],P_4[\overline{K}_{n-3}^1], P_4[K_{n-3}^2],P_4[\overline{K}_{n-3}^2] \}\cup\{P_4[K_r^1,K_{n-r-2}^2] : 1\leq r\leq n-3\}\cup \{P_4[\overline{K}^1_r,\overline{K}^3_{n-r-2}]: 1\leq r\leq n-3 \} $.
	\end{itemize}
\end{theorem}
\begin{proof}
	For $k=2$, let $G$ be a graph such that $G$ and $\overline{G}$ are connected graphs. From Theorem~\ref{th3.1} $(a)$,  $\gamma^r_2(G)=1$ if and only if $G$ is either $K_2$ or $P_3$.  Then both $G$ and $\overline{G}$ have distance $2$-resolving domination number greater or equal to $2$. Hence, $\gamma^r_2(G)+\gamma^r_2(\overline{G})\geq 4$ and $\gamma^r_2(G)\cdot\gamma^r_2(\overline{G})\geq 4$. Also based on Proposition~\ref{prop2.3} we have $\gamma^r_2(P_4)=\gamma^r_2(\overline{P}_4)=2$, then the lower bounds are sharp. \par 
	Otherwise, we have from Theorem~\ref{th3.1} $(c)$, $K_n$ is the only connected graph with distance $2$-resolving domination number equal to $n-1$. Since the complement of the complete graph is disconnected, it follows that $\gamma^r_2(G)\leq n-2$. Moreover, from Theorem~\ref{th3.1} $(b)$, for $n\geq 4$, $\gamma^r_2(G)=n-2$ if and only if $G$ is either $P_4$, $K_{s,t} (s,t\geq 1)$, $K_s+\overline{K}_{t} (s\geq 1, t\geq 2)$, or $K_s+(K_1\cup K_{t}) (s,t\geq 1)$. The only graph from these graphs whose complement graph is also connected is the path $P_4$. Since $P_4$ is self-complementary, we can conclude that  $\gamma^r_2(G)+\gamma^r_2(\overline{G})\leq 2n-4$ and $\gamma^r_2(G)\cdot\gamma^r_2(\overline{G})\leq (n-2)^2$, where the equality holds if and only if $G\cong P_4$. \par 
	For $k\geq 3$, we have $\gamma^r_k(P_4)=1$. The graph $P_4$ is self-complementary and is the only graph in Theorem~\ref{th3.1} $(a)$ whose complement is a path graph having $\gamma^r_k(\overline{G})=1$. Then $ \gamma^r_k(G)+\gamma^r_k(\overline{G})\geq 2$ and $ \gamma^r_k(G)\cdot \gamma^r_k(\overline{G})\geq 1$, and these lower bounds are achieved if and only if $G$ is $P_4$. \par
	Otherwise, we have from Theorem~\ref{th3.1} $(c)$, $\gamma^r_k(G)=n-1$ if and only if $G$ is a complete graph. It follows that $\gamma^r_k(G)\leq n-2$. Furthermore, in Theorem~\ref{th3.1} $(b)$,  $\gamma^r_k(G)=n-2$ if and only if $G$ is either $K_{s,t} (s,t\geq 1)$, $K_s+\overline{K}_{t} (s\geq 1, t\geq 2)$, or $K_s+(K_1\cup K_{t}) (s,t\geq 1)$. Since the complements of these graphs are all disconnected, it follows that $\gamma^r_k(G)\leq n-3$ and $\gamma^r_k(\overline{G})\leq n-3$.   Therefore, for $k\geq 3$,  $\gamma^r_k(G)+\gamma^r_k(\overline{G})\leq 2n-6$ and $\gamma^r_k(G)\cdot \gamma^r_k(\overline{G})\leq (n-3)^2$. The only connected graph of order $4$ whose complement graph is also a connected graph is $P_4$, we have $\gamma^r_2(P_4)=\gamma^r_2(\overline{P}_4)=1$. Also for $n\geq 5$, based on Lemma~\ref{lem2.2}, we have $\gamma^r_k(G)=\gamma^r_k(\overline{G})=n-3$ if and only if  $dim(G)=dim(\overline{G})=n-3$. It follows that $\gamma^r_k(G)+\gamma^r_k(\overline{G})=2n-6$ if and only if $dim(G)+dim(\overline{G})=2n-6$. In \cite{her}, if $G$ and $\overline{G}$ are both conncted graphs, we have $dim(G)+dim(\overline{G})=2n-6$ if and only if $G\in \{P_4,C_5,B\}\cup\{P_4[K_{n-3}^1],P_4[\overline{K}_{n-3}^1], P_4[K_{n-3}^2],P_4[\overline{K}_{n-3}^2] \}\cup\{P_4[K_r^1,K_{n-r-2}^2] : 1\leq r\leq n-3\}\cup \{P_4[\overline{K}^1_r,\overline{K}^3_{n-r-2}]: 1\leq r\leq n-3 \} $.
\end{proof}

\section{Concluding remarks} 
The study of the distance $k$-resolving domination number could be extended to other graph families and operations on graphs not discussed here. For example for trees, a formula in \cite{hening3} is provided to compute efficiently $\gamma^r(T)$ for any tree $T$. We ask if it would be possible also for $\gamma^r_k(T)$ when $k\geq 2$. \par 
For $k\geq 1$, we denote $N_k(v)=\{x\in V : 0 < d_G(v,x)\leq k \}$, the \emph{open $k$-neighborhood} of a vertex $v$ in $V$. The \emph{$k$-locating-dominating set} defined as a set $X\subseteq V$, verifying for every $v,u\in V \setminus X$, we have $\emptyset \neq N_k(v)\cap X\neq N_k(u)\cap X \neq~\emptyset$. The minimum cardinality of such set is called the \emph{$k$-locating-domination number} denoted by $LD_k(G)$. Results about the $k$-locating-domination number can be found surveyed in  \cite{lob}. Necessarily every $k$-locating-dominating set is a distance $k$-resolving dominating set, the opposite is not true. Therefore for all $k\geq 1$, we have  $\gamma^r_k(G)\leq LD_k(G)$.  For $k=1$, in \cite{cacer} it is shown that $LD_1(T)\leq 2\gamma^r(T)-2$ for any tree $T$ different from $P_6$. In \cite{gonz}, it is proved that  $LD_1(G)\leq (\gamma^r(G))^2$ for any graph $G$ not containing $C_4$ or $C_6$ as a subgraph. Finding an upper bound for $LD_1(G)$ in terms of $\gamma^r(G)$ for graphs in general is still open, it is shown \cite{gonz} that such an upper bound is at least exponential in terms of $\gamma^r(G)$. Is it possible to find upper bounds for $LD_k(G)$ in terms $\gamma^r_k(G)$ when $k\geq 2$ for graphs?

\section*{Acknowledgement}
We gratefully acknowledge Airlangga University, Indonesia, and Ibn Zohr University, Morocco, for their supervision and support in accomplishing this work.

\end{document}